\documentclass[12pt]{article}
\usepackage{bbold,dsfont}
\usepackage{diagbox}
\usepackage{amsmath}
\usepackage[normalem]{ulem}

\renewcommand{\emph}[1]{\textit{#1}}
\usepackage{enumerate,amsmath,amsthm,latexsym,amssymb}
\usepackage{color}\usepackage{graphicx}

\definecolor{brown}{cmyk}{0, 0.72, 1, 0.45}
\definecolor{grey}{gray}{0.5}

\newcommand{\old}[1]{}

\newcounter{rot}

\newcommand{\ignore}[1]{}

\def\bM{{\bf M}}

\def\cA{{\mathcal A}}
\def\cB{{\mathcal B}}

\newcommand{\set}[1]{\left\{#1\right\}}

\def\hc{\widehat{c}}

\newcommand{\proofend}{\hspace*{\fill}\mbox{$\Box$}\\ \medskip\\ \medskip}

\def\ii_(#1,#2){i_{#1}^{#2}}

\allowdisplaybreaks

\def\cM{\mathcal{M}}
\def\a{\alpha}
\def\b{\beta}

\def\e{\varepsilon}

\def\g{\gamma}

\def\z{\zeta}

\def\th{\theta}

\def\r{\rho}

\newcommand{\rdup}[1]{\left\lceil #1 \right\rceil}
\newcommand{\rdown}[1]{\mbox{$\left\lfloor #1 \right\rfloor$}}

\newcommand{\ra}{\longrightarrow}
\newcommand{\rai}{\ra \infty}

\newcommand{\ooi}{(1+o(1))}

\def\cE{\mathcal{E}}

\newcommand{\brac}[1]{\left( #1 \right)}

\def\cH{\mathcal H}

\def\E{{\bf E}}

\renewcommand{\Pr}{\operatorname{\bf Pr}}
\newcommand\bfrac[2]{\left(\frac{#1}{#2}\right)}

\def\bC{\bar{C}}

\newtheorem{theorem}{Theorem}[section]

\newtheorem{lemma}[theorem]{Lemma}
\newtheorem{corollary}[theorem]{Corollary}

\newcounter{thmtemp}

\newcommand{\nospace}[1]{}

\def\path{\operatorname{PATH}}

\newcommand{\beq}[2]{\begin{equation}\label{#1}#2\end{equation}}
\newcommand{\mult}[2]{\begin{multline}\label{#1}#2\end{multline}}

\def\bA{{\bf A}}

\def\rank{\mathrm{rank}}

\def\bC{{\bf C}}

\def\bc{{\bf c}}

\def\bM{{\bf M}}

\newcommand{\Ind}{\mathds{1}}
\begin{document}
\title{On the rank of a random binary matrix}
\author{Colin Cooper\thanks{Research supported in part by EPSRC grant EP/M005038/1}\and
Alan Frieze\thanks{Research supported in part by NSF Grant DMS1661063}\and Wesley Pegden\thanks{Research supported in part by NSF grant DMS1363136}}
\maketitle

\begin{abstract}
We study the rank of a random $n \times m$ matrix $\bA_{n,m;k}$ with entries
from $GF(2)$, and exactly $k$ unit entries in each column, the other
entries being zero. The columns are chosen independently and uniformly at random from the
set of all ${n \choose k}$ such columns.

We obtain an asymptotically correct estimate for the rank as a
function of the number of columns $m$ in terms of $c,n,k$, and where
$m=cn/k$. The matrix $\bA_{n,m;k}$ forms the vertex-edge incidence matrix
of a $k$-uniform random hypergraph $H$. The rank of $\bA_{n,m;k}$ can be
expressed as follows. Let $|C_2|$ be the number of vertices of the
2-core of $H$, and $|E(C_2)|$ the number of edges. Let $m^*$ be the value
of $m$ for which $|C_2|= |E(C_2)|$. Then w.h.p. for $m<m^*$ the rank of
$\bA_{n,m;k}$ is asymptotic to $m$, and for $m \ge m^*$ the rank is
asymptotic to $m-|E(C_2)|+|C_2|$.

In addition, assign i.i.d. $U[0,1]$ weights $X_i, i \in {1,2,...m}$ to
the columns, and define the weight of a set of columns $S$ as
$X(S)=\sum_{j \in S} X_j$. Define a basis as a set of $n-\Ind (k\text{ even})$
linearly independent columns. We obtain an asymptotically correct
estimate for the minimum weight basis. This generalises the
well-known result of Frieze [On the value of a random minimum
spanning tree problem, Discrete Applied Mathematics, (1985)] that, for
$k=2$,   the expected length of a minimum weight spanning
tree tends to $\zeta(3)\sim 1.202$.
\end{abstract}

\section{Introduction}
Let $\Omega_{n,k}$ denote the set of  vectors of length $n$, with $0,1$ entries, with exactly $k$ 1's, all other entries being zero.  The addition of entries is over the field $GF_2$, i.e., the vector addition is over  $(GF_2)^n$. Let $\bA_{n,m;k}$ be the random $n\times m$  matrix where the  columns form a random $m$-subset of $\Omega_{n,k}$.

In a recent paper \cite{minors},  we studied the binary matroid $\cM_{n,m;k}$ induced by the columns of $\bA_{n,m;k}$. It was shown that for any fixed binary matroid $M$, there were constants $k_M,L_M$ such that if $k\geq k_M$ and $m\geq L_mn$ then w.h.p.  $\cM_{n,m;k}$  contains $M$ as a minor.
The paper \cite{minors} contributes to the theory of random matroids as developed by  \cite{AY}, \cite{BPP}, \cite{KL}, \cite{KO}, \cite{OSWW}.
In this paper we study a related aspect of $\bA_{n,m;k}$, namely its  rank, and improve on results from Cooper \cite{Co1}. As a consequence of the precise estimate of rank in
Theorem \ref{th1}  we can   give an expression, \eqref{exp}, for the solution value of the following optimization problem.

Suppose that we assign i.i.d. $U[0,1]$ weights $X_\bc$ to the vectors $\bc \in\Omega_{n,k}$ and let the weight of a set of columns $S$ be $X(S)=\sum_{\bc\in S}X_\bc$. Define a {\em basis} as a set of $n-\Ind (k\text{ even})$ linearly independent columns. What is the expected weight $W_{n,k}$ of a minimum weight basis? When $k=2$ this amounts to estimating the expected length of a minimum weight spanning tree of $K_n$ which has the limiting value of $\z(3)$, see Frieze \cite{F1}.

Our result on the rank of $\bA_{n,m;k}$ takes a little setting up. Let $H=H_{n,m;k}$ denote the random $k$-uniform hypergraph with vertex set $[n]$ and $m$ random edges taken from $\binom{[n]}{k}$. There is a natural bijection between $\bA_{n,m;k}$ and $H_{n,m;k}$ in which column $\bc$ is replaced by the set $\set{i:\bc_i=1}$. The $\r$-core of a hypergraph $H$ (if it is non-empty) is the maximal set of vertices that induces a sub-hypergraph of minimum degree $\r$. The 2-core $C_2=C_2(H)$ plays an important role in our first theorem.

\subsection{Matrix Rank}
{\bf Notation:} We write $X_n\approx Y_n$ for sequences $X_n,Y_n,n\geq 0$ if $X_n=(1+o(1))Y_n$ as $n\to\infty$. We will use some results on the 2-core of random hypergraphs. The size of of the 2-core has been asymptotically determined, see for example Cooper \cite{CH} or Molloy \cite{M}; we recall the basic w.h.p. results here.
In random graphs $G_{n,m}=H_{n,m;2}$ the 2-core grows gradually with  $m$ following the emergence of the first cycle of size $O(\log n)$. For $k \ge 3$, the 2-core is either empty or of linear size and emerges around some threshold value $\widehat m_k$. Initially above $\widehat m_k$ the 2-core has more vertices than edges, and there is a larger value  $m^*$, around which the number of vertices and edges becomes the same. Below $m^*$ the rank of the 2-core grows asymptotically as the number of edges, and above $m^*$ as the number of vertices.

To describe the size of the 2-core, we parameterise $m$ as $m=cn/k,\,c=O(1)$ and consider the equation
\beq{cstar}{
x=(1-e^{-cx})^{k-1}.
}
For $k\geq 3$, define $\hc_k$  by
\[
\hc_k=\min\set{c: x=(1-e^{-cx})^{k-1} \text{ has a solution  }x_c \in (0,1]}.
\]
It is known that $c<\hc_k$ implies that $C_2=\emptyset$. If
$c>\hc_k,\;c=O(\log n)$, let $x_c$ be the largest solution to \eqref{cstar} in $[0,1]$. Then q.s.\footnote{A sequence $\cE_n$ of events occurs {\em quite surely} (q.s.) if $\Pr(\neg\cE_n)=O(n^{-C})$ for any constant $C>0$.}
\begin{align}
&\left| \,|C_2|-n(x_c^{1/(k-1)}-cx_c+cx_c^{k/(k-1)}) \right|\leq n^{3/4},\label{sizeC2}\\
&\left| \,|E(C_2)|-n(cx_c^{k/(k-1)}/k) \right|\leq n^{3/4}.\label{sizeC2a}
\end{align}
Let $c_k^*$ be the value of $c$ for which the 2-core has asymptotically the same number of vertices and edges. More precisely, we use \eqref{sizeC2} and \eqref{sizeC2a} to define $c_k^*$ by
\beq{defck}{
c_k^*:=\min\set{c\geq \hc_k:x_c^{1/(k-1)}-cx_c+cx_c^{k/(k-1)}=\frac{cx_c^{k/(k-1)}}{k}}.
}
Define $m^*$ by $m_k^*=c_k^*n/k$.
We will prove,
\begin{theorem}\label{th1}
If $m=O(n)$ then w.h.p.
$$\rank(\bA_{n,m;k})\approx \begin{cases}|E(H)|&m<m_k^*.\\
|E(H)|-|E(C_2)|
+|C_2|&m\geq m_k^*. \end{cases}$$
\end{theorem}
Note that when $k=2$ we have $c_2^*=0$ and the theorem follows from the fact that an isolated tree with $t$ edges induces a sub-matrix of rank $t$ in $\bA_{n,m;k}$. We therefore concentrate on the case $k\geq 3$.

Using \eqref{sizeC2}, we can express Theorem \ref{th1} directly in terms of $c$ by
\begin{corollary}\label{cor1}
Suppose that $k\geq 3$ and $m=cn/k$. Then, w.h.p.
\beq{exp}{
\rank(\bA_{n,m;k})\approx \begin{cases}m&c< c_k^*.\\ m-m{x_c}^{k/(k-1)}+n(x_c^{1/(k-1)}-cx_c+cx_c^{k/(k-1)})&c\geq c_k^*.\end{cases}
}
\end{corollary}

Around  $m=n( \log n+c_n)/k$  the remaining vertices of degree one in $H$ disappear, and $\bA_{n,m;k}$ has full rank up to parity, i.e., $\rank(\bA_{n,m;k})=n^*$ where
$$n^*=n-\Ind (k\text{ even}).$$
\begin{theorem}\label{th2}
Suppose that $k\geq 3$.
\begin{enumerate}[(i)]
\item Given a constant $A>0$, there exists $\g=\g(A)$ such that for $m\geq\g n\log n$,
$$\Pr(\rank(\bA_{n,m;k})<n^*)=o(n^{-A}).$$
\item If $m=n(\log n+c_n)/{k}$ then
$$\lim_{n\to \infty}\Pr(\rank(\bA_{n,m;k}=n^*))=\begin{cases}0&c_n\to-\infty\\
e^{-e^{-c}}&c_n\to c\\ 1 &c_n\to+\infty.\end{cases}$$
\end{enumerate}
\end{theorem}
We can easily modify the proof of part (ii) of Theorem \ref{th2} to give the following hitting time version. Suppose that we randomly order the columns of $\bA_{n,M;k}$ where $M=\binom{n}{k}$. Let $\bM_m$ denote the matrix defined by the first $m$ columns in this order.
$$m_1=\min\set{m:\bM_m\text{ has $n^*$ non-zero rows}}\text{ and let }m^*=\min\set{m:\bM_m\text{ has rank $n^*$}}.$$
\begin{theorem}\label{th2a}
$m_1=m^*$ w.h.p.
\end{theorem}
Some time after completion of this manuscript, we learnt from Amin Coja-Oghlan that he has an independent proof of Theorem \ref{th1}.
\subsection{Minimum Weight Basis}
The expression \eqref{exp} enables us to estimate the expected optimal value to the minimum weight basis problem defined above. Suppose that we assign i.i.d. $U[0,1]$ weights $X_\bc,\bc\in \Omega_{n,k}$  to the $|\Omega_{n,k}|=\binom{n}{k}$ distinct vectors with exactly $k$  unit entries, all other entries being zeroes. The weight of a set of columns $C$ is $X(C)=\sum_{\bc\in C}X_\bc$. Let $W_{n,k}$ be the minimum weight of any basis  of $n^*=n-\Ind (k\text{ even})$ linearly independent columns, chosen from the $\binom{n}{k}$ column vectors $\bc\in \Omega_{n,k}$.
Define the random matrix $\bA_{n,p;k}$  to consist of the vectors $\bc \in \Omega_{n,k}$ with weight $X_{\bc}$ at most $p$.

We show in Section \ref{mwb} below that if $W_{n,k}$ denotes the weight of a minimum weight basis then
\beq{Int1}{
\E(W_{n,k})=\int_{p=0}^1(n^*-\E(\rank(\bA_{n,p;k})))dp.
}

Corollary \ref{cor1} and Theorem \ref{th2} can be substituted into \eqref{Int1} to yield an asymptotic formula for $W_{m,k}$.
\begin{theorem}\label{WTF}
 Let $x=x(c)$ be the largest solution of $x=(1-e^{-cx})^{k-1}$ in $(0,1]$, then
 \begin{equation}
 \frac{n^{k-2}}{(k-1)!}\E(W_{n,k})\approx c_k^*\brac{1-\frac{c_k^*}{2k}}+
 \int_{c_k^*}^{\infty} \brac{ e^{-cx}\brac{1+\frac{(k-1)cx}{k}}-\frac{c}{k}(1-x)}dc
 \label{opaque}
 \end{equation}
\end{theorem}

We note the remarkable fact that, by  the result of Frieze \cite{F1},
for $k=2$ and with $c_2^*=0$,
the expression in \eqref{opaque} must equal $\zeta(3)$.  We have numerically estimated the first few values  as a function of  $k$:

\vspace{-0.1in}
\begin{center}
  \renewcommand{\arraystretch}{1.5}
\begin{tabular}{c|c|c|c|c|c|c|c|c|c}
$k$&2&3&4&5&6&7&8&9&10\\
\hline
$\frac{n^{k-2}}{(k-1)!}\E(W_{n,k})$&$\z(3)\approx 1.202$&1.563&2.021&2.507&3.003&3.501&4.000&4.500&5.000\\
\end{tabular}
\end{center}
It appears the values are getting close to $k/2$ as $k$ grows, and  this is indeed the case.

\begin{theorem}\label{th3}
For $k \ge 3$, and some $\e_k$, $|\e_k| \le 5$,
\beq{largek}{
\lim_{n \rai}\frac{n^{k-2}}{(k-1)!}\E(W_{n,k})= \frac{k}{2}\brac{1+\e_k e^{-k}}.}
\end{theorem}

\section{Matrix Rank}\label{MR}

We study the random matrix $\bA_m$ distributed as $\bA_{n,m;k}$, with corresponding hypergraph $H_m$ distributed as $H_{n,m;k}$.
We let $c=km/n$.

The first step of our proof is to ``peel off'' edges of the hypergraph $H_m$, and thus columns of the matrix $\bA_m$, containing vertices of degree $1$.

In particular, we set $H_m:=H_m$, and then, recursively, so long as $H_i$ contains a vertex $x_i$ of degree $1$, then for the edge $e_i\ni x_i$ in $H_i$, we set
\begin{align*}
  E(H_{i-1})&=E(H_{i})\setminus \{e_i\}\\
  V(H_{i-1})&=V(H_i)\setminus\{x\in e_i\mid \deg_{H_i}(x)=1\}.
\end{align*}

In a corresponding sequence $\{\bA_i\}$ beginning from $\bA_m$, we obtain $\bA_{i-1}$ from $\bA_i$ by removing the column $c_i$ corresponding to $e_i$, and the (at least one) rows whose only $1$s were in that column.  Note that for all $i<m$ for which $\bA_i$ is defined, we have
\[
\rank(\bA_i)=\rank(\bA_{i+1})-1.
\]
This recursion terminates at
\begin{equation}\label{c2}
  \bC_2=\bA_{m_2},
\end{equation}
where $m_2$ is the number of edges in the the 2-core of the hypergraph $H$, and moreover, we have that $H_{m_2}$ is precisely the 2-core of $H$.  Thus we have that
\begin{equation}
\rank(\bA_m)=m_1+\rank(\bC_2).
\end{equation}

In particular, we consider cases which control the behavior of the rank of the 2-core $\bC_2=H_{m_2}$ of $H$.\\
\textbf{Case 1:} $c<c_k^*$.\\
In this case we appeal to a result of Pittel and Sorkin \cite{PS}.  In particular, the columns associated with the 2-core $\bC_2$ are distributed as uniformly random, subject to each vertex/row of the 2-core being in at least two columns. It follows from Theorem 2 of Pittel and Sorkin \cite{PS} that the rank of the columns $\bc_{m_1+1},\bc_{m_1+2},\ldots,\bc_m$ is $\approx m_2=m-m_1$. (The Theorem 2 in \cite{PS} is stated in transpose to the formulation given here.)

For this case the first claim of \eqref{exp}, and Theorem \ref{th1}, have been verified.

{\bf Case 2: $c\geq c_k^*$.} \\
To prove Theorem \ref{th1} for $c\geq c_k^*$ we only need to verify that w.h.p.
\beq{CC2}{
\rank(\bC_2)\approx |V(C_2)|.
}

In this case we need some basic facts about hypergraphs. We say a hypergraph $H$ is {\em linear} if edges only intersect in at most one vertex. We define a $k$-uniform {\em cactus} as follows. A single edge is a cactus. An $(\ell+1)$-edge  cactus $C'$ is the structure  obtained from an $\ell$-edge cactus $C$ with vertex set $V(C),|V(C)|=(k-1)\ell+1$ as follows. Choose $x\in V(C)$ and let $V(C')= V(C) \cup \set{v_1,...v_{k-1}}$ where $\set{v_1,...v_{k-1}}$ is disjoint from $V(C)$.
The edge set $E(C')$ of $C'$ is $E(C) \cup \set{e'}$ where $e'=\set{x,v_1,...v_{k-1}}$. We need the following simple lemma.
\begin{lemma}\label{simple}
A connected $k$-uniform simple hypergraph $C$ with no cycles is a cactus.
\end{lemma}
\begin{proof}
This can easily be verified by induction. We simply remove one terminal edge $e=\set{v_1,v_2,\ldots,v_k}$ of a longest path $P$. We can assume here that $v_2,\ldots,v_k$ are all of degree one, else $P$ can be extended. Deleting $e$ gives a new connected hypergraph $C'$ which is a cactus by induction.
\end{proof}
For a $k$-uniform linear hypergraph $H$ let $L(H)=(k-1)|E(H)|+1$.
\begin{lemma}\label{key}
Let $H$ be a connected $k$-uniform linear hypergraph.
\begin{enumerate}[(a)]
\item \label{p.inequality} $|V(H)| \le L(H)$.
\item \label{p.equality} $|V(H)|=L(H)$ if and only if $H$ does not contain any cycles.
\item \label{p.deleting} By deleting at most $L(H)-|V(H)|$ edges we can create a subgraph $H'$ with $V(H')=V(H)$ and no cycles.
\end{enumerate}
\end{lemma}
\begin{proof}

We consider two cases:\\
 \textbf{Case 1:} $H$ contains no cycles.\\
In this case, we consider a longest path of edges in $H$; that is consider a longest sequence $e_1,e_2,\dots,e_\ell$ such that for each $1<i<e_\ell$, $e_i$ intersects $e_{i-1}$, $e_{i+1}$, and no other edges in the sequence.  Since the path is longest and $H$ has no cycles, we know that $e_\ell$ intersects no edge in $H$ other than $e_{\ell-1}$.

In particular, we define a hypergraph $H'$ with $E(H')=E(H)\setminus \{e_\ell\}$ and $V(H')=V(H)\setminus (e_\ell\setminus e_{\ell-1}).$  $H'$ has one fewer edge and $k-1$ fewer vertices than $H$, so we have $L(H)=|V(H)|$ by induction, proving the Lemma for this case.\\
\textbf{Case 2:} $H$ contains a cycle $C$.\\
In this case, we consider an edge $e$ in a cycle $C$ of $H$.  Removing the edge $e$ leaves a hypergraph on the same vertex set with one fewer edge and with at most $k-1$ connected components (counting isolated vertices as connected components).  Applying the Lemma inductively to each component, we see that the sum of $L(H_i)$ over the $(k-1)$ components $H_i$ of $H\setminus e$ satisfies
  \[
  \sum_{i=1}^{k-1} L(H_i)\leq L(H)-(k-1)+(k-2)\leq L(H)-1,
  \]
  since removing $e$ decreases the sum by $k-1$, while the additive term in the definition of $L(H)$ inflates the sum by at most $(k-2)$ (as the number of components has increased by up to $k-2$).  On the other hand we of course have

  \[
  \sum_{i=1}^{k-1} |V(H_i)|=|V(H)|.
  \]
We now apply parts \eqref{p.inequality} and \eqref{p.deleting} of the Lemma to each component by induction, and conclude that the Lemma does hold for $H$.
\end{proof}

In the following lemma we prove a property of $H_{n,m;k}$. It will be more convenient to work with $H_{n,p;k}$ where $m=\binom{n}{k}p$. We  use the fact that for any hypergraph property $\cH$ that is monotone increasing or decreasing with respect to adding edges,
\beq{relateH}{
\Pr(H_{n,m;k}\in\cH)\leq O(1)\Pr(H_{n,p;k}\in\cH).
}
This is well-known for graphs and is essentially a property of the binomial random variable, $E(H_{n,p;k})$,  the number of edges of $H_{n,p;k}$.

Similarly, if $\cA$ is a matrix property that is monotone increasing or decreasing with respect to adding columns, then
\beq{relate}{
\Pr(\bA_{n,m;k}\in\cA)\leq O(1)\Pr(\bA_{n,p;k}\in\cA).
}

\begin{lemma}\label{S-L}
Suppose that $m=O(n\log n)$.
\begin{enumerate}[(a)]
\item With probability $1-o(n^{-1})$, for every set of vertices $S$ of size $\ell_0=\log^{1/2}n\leq s\leq s_0=n^{1-\a}$ we have that
  $L(S)\leq s+\rdown{\th s},$
  where $\th=\frac{1}{\log^{1/4}n}$.  Here $H[S]$ is the hypergraph of edges belonging completely to $S$.
\item Then w.h.p., there are at most $n^{o(1)}$ vertices in cycles of size at most $\log^{1/2}n$.
\end{enumerate}
\end{lemma}
\begin{proof}
(a) We can use \eqref{relateH} here with $p=\frac{C\log n}{n^{k-1}}$ for some constant $C>0$. Let $s_1=s+\rdown{\th s}+1$. The expected number of sets failing this property can be bounded by
\begin{align}
&\sum_{s=\ell_0}^{s_0}\binom{n}{s}\sum_{L\geq s_1}\binom{\binom{s}{k}}{L/(k-1)}\brac{\frac{C\log n}{n^{k-1}}}^{L/(k-1)} \nonumber\\
&\leq\sum_{s=\ell_0}^{s_0}\bfrac{ne}{s}^s \sum_{L\geq s_1}\bfrac{Ce^2s^{k}\log n(k-1)}{k!Ln^{k-1}}^{L/(k-1)}\nonumber\\
  &\leq\sum_{s=\ell_0}^{s_0}\sum_{L\geq s_1}(Ce^3\log n)^{L}\bfrac{s}{n}^{L-s} \bfrac{s}{L}^{L/(k-1)}\nonumber\\
  &\leq \sum_{s=\ell_0}^{s_0}\sum_{L\geq s_1}\brac{(Ce^3\log n)\bfrac{s}{n}^{1-s/L}}^{L}\label{oo(1)}
\end{align}

Let $u_{s,L}$ denote the summand in \eqref{oo(1)}. Then we have
\begin{align*}
u_{L,s}\leq \brac{(Ce^3\log n)^{2\a^{-1}}{\displaystyle \bfrac{s}{n}}^{\th}}^s&\leq n^{-(\a-o(1))\th s}&L\leq 2\a^{-1}s.\\
u_{L,s}\leq \brac{(Ce^3\log n){\displaystyle \bfrac{s}{n}^{1-\a/2}}}^L&\leq n^{-(1-o(1))\a L/2} &L>2\a^{-1}s.
\end{align*}
Thus,
\begin{align}
\sum_{s\geq 2}\sum_{L\geq s_1}u_{s,L}&\leq \sum_{s=\ell_0}^{s_0}\sum_{L=s+\rdup{\th s}}^{2\a^{-1}s}n^{-(\a-o(1))\th s}+ \sum_{s=\ell_0}^{s_0}\sum_{L\geq 2\a^{-1}s}n^{-(1-o(1))\a L/2}\nonumber\\
&\leq2\a^{-1}s_0\sum_{s=\ell_0}^{s_0}n^{-(\a-o(1))\th s}+ \sum_{s=\ell_0}^{s_0}n^{-(1-o(1))s/2}\nonumber\\
&=o(n^{-1}).\label{o(1)}
\end{align}
(b) The expected number of vertices in small cycles can be bounded by
$$
\sum_{\ell=2}^{\log^{1/2}n}\binom{n}{(k-1)\ell}((k-1)\ell)!p^\ell\leq  \sum_{\ell=2}^{\log^{1/2}n}(n^{k-1}p)^\ell\leq \sum_{\ell=2}^{\log^{1/2}n}(C\log n)^\ell=n^{o(1)}.$$
Part (b) now follows from the Markov inequality.
\end{proof}

\subsection{Growth of the mantle}
We now consider the change in the rank of the sub-matrix $\bC_2$ of the edge-vertex incidence matrix $\bA_m$ (see \eqref{c2}) corresponding to the 2-core of the column hypergraph, caused by adding a column to $\bA_m$.  In this section, we will assume in our calculations that no two edges share more than one vertex, and that the 2-core consists of a single connected component.  This does not affect our asymptotic analysis because simple first-moment calculations show that:
\begin{enumerate}
\item There are only a bounded number of edges sharing more than one vertex, and
\item Any subset of the random hypergraph of minimum degree must be of linear size; together with \eqref{conc}, below, this then implies that the 2-core can only have one connected component in the present regime, since the appearance of another component at any state would increase the size of the 2-core by too much.
\end{enumerate}

So suppose now that the addition of $e$ increases the size of the 2-core.
Let $A$ denote the set of additional vertices and $F$ denote the set of additional edges added to $C_2$ by the addition of $e$, where $A\subset V(F)$. We include $e$ in $F$. 

We remark first that with $c,x$ as in \eqref{cstar}, that q.s.
\beq{conc}{
|C_2|-n(1-e^{-cx}(1+cx))|\leq n^{3/4}, \qquad \text{and} \qquad |E(C_2)|-mx^{k/(k-1)}|\leq n^{3/4}.
}
Therefore we can assume that adding an edge to $\bA_m$ can only increase $C_2,\;E(C_2)$ by at most $n^{3/4}$.
We  use Lemma \ref{S-L} with $\a=3/4$ in our discussion of the hypergraph $F$.

Obviously the increase in rank from adding $F$ to the 2-core is bounded above by the size of the vertex-set $A$.  To bound it from below, we proceed as follows:

{\bf Case 1:} First consider the case where there are no cycles in $F$.  We will show that the rank increases by precisely the number of new vertices.

Let $|A|=k$.  We will define an ordering $a_1,\dots,a_k$ of $A$ and a corresponding ordering $f_1,\dots,f_k$ of a subset of $F$.  To begin, we claim there must exist $v\in A$ and $v\in f\in F$, $f\ne e$, such that $f\setminus\set{v}\subseteq C_2$. For this consider a longest path $e_1,\dots,e_\ell$ of edges in $F$.  Since the hypergraph is simple and contains no cycles, we have that $e_\ell\cap (\bigcup_{i=1}^{\ell-1} e_i)=e_\ell\cap e_{\ell-1}=\{v\}$ for some single vertex $v$.  On the other hand, all vertices of $e_\ell$ must have degree 2 in $F\cup C_2$, and so $e_\ell\setminus v$ must lie entirely in $C_2$.  We set $f_1=e_\ell$, $a_1=v$, and then we remove $f_1$ from $F$ and $a_1$ from $A$, defining $C_2^1=C_2\cup f_1$ (though it is not a two-core of any hypergraph), and apply induction to obtain the sequences $a_1,\dots,a_k$, $f_1,\dots,f_k$, and the corresponding sequence $C_2^i$ defined by $C_2^0=C_2$, and $C_2^{i+1}=C_2^i\cup f_{i+1}$.

These sequences have the property that
\[
\rank(C_2^{i+1})=\rank(C_2^i)+1,
\]
since the edge $f_i$ added to $C_2^i$ in step $i+1$ contains exactly one vertex outside of $C_2^i$.  (In the matrix, we are adding a column containing a 1 in a row which previously had no 1's).

In particular, the rank in this case increases by exactly the size of $A$.

{\bf Case 2:} The total contribution to the rank of the 2-core in $m=O(n \log n)$ steps from the case where $F$ contains a cycle of length at most $\log^{1/2}n$ can be bounded by $n^{3/4+o(1)}$. This follows from Lemma \ref{S-L}(b) and \eqref{conc}. This is negligible, since the core has size $\Omega(n)$ in the regime we are discussing.

{\bf Case 3:} Suppose that $F$ contains cycles of size at least $\log^{1/2}n$ which we remove by deleting $s$ edges. When we do this we may lose up to $ks$ vertices from $A$. Let the resulting vertex set be $A'$ and edge set be $F'$. Up to $ks$ vertices of $A'$ may have degree 1. Attach these vertices to $C_2$ using disjoint edges to give edge set $F''$. All vertices of $A'$ now have degree at least 2 in $F''$ and $F''$ has no cycles. According to the argument in Case 1, the increase in rank due to adding $F''$ is $|A'|\ge |A|-ks$ and this is at most $ks$ larger than the increase in rank due to adding $F'$. Thus the increase in rank due to adding $F\supseteq F'$ is at least $|A|-2ks$  and at most $|F| \le |A|+s+1$. It follows from Lemma \ref{key}(c) and Lemma \ref{S-L}(a) that $s=o(|A|)$.

In summary we find that if $m=O(n\log n)$ and $m \ge c^*n/k$ then, with probability $1-o(n^{-1})$, the rank of $\bC_2$ satisfies
\beq{end}{
\brac{1-o(1)}|C_2|\leq \rank(\bC_2)\leq |C_2|.
}
The upper bound follows because the rank of $\bC_2$ is at most the number of rows in $\bC_2$. This proves \eqref{CC2}. To finish the proof of Theorem \ref{th1} we require that \eqref{end} remains true if we take expectations. For this we  use the error probability of $o(n^{-1})$ in \eqref{o(1)}.

\subsection{Proof of Theorem \ref{th2}}
{\bf Proof of part (i):}\\
Given a set of rows $S$, the number of choices of column  (distinct edges) that have an odd number of non-zero entries in $S$ is
$$T_{s,k}=\binom{s }{1}\binom{n-s}{k-1}+\binom{s}{3}\binom{n-s}{k-3}+\cdots+\binom{s}{k}.$$
If $\rank(\bA_{n,p;k})<n^*$ then there exists a set $S$ of rows such that (i) each column of $\bA_{n,p;k}$ has an even number of non-zero entries $j$ in $S$ and (ii) $|S|\leq n^*$. For a fixed $S$, denote this event by $\cB_S$ and note that it is monotone decreasing. Then
\beq{eq1}{
\Pr(\cB_S)=(1-p)^{T_{s,k}}.
}
For $s \ge k$,
\[
T_{s,k} \ge \binom{s }{1}\binom{n-s } {k-1}+\binom{s}{k}= \frac{s}{(k-1)!} \brac{\frac{s^{k-1}}{k}+(n-s)^{k-1}} \ooi
\]
The bracketed term on the right hand side is minimized when $s=\a n$ where
$\a=
k^{1/(k-2)}/(1+k^{1/(k-2)})$. Let $\b_k=(\a^{k-1}/k+(1-\a)^{k-1})$ then
\[
T_{s,k} \ge \b_k s \frac{n^{k-1}}{(k-1)!} \ooi.
\]
We can choose $p=\frac{(A+2)\log n}{\b_k\binom{n-1}{k-1}}$ and then use monotonicity of rank as a function of $p$ to claim the result for larger $p$.
\begin{align}
\Pr(\exists S:\cB_S\text{ occurs})
&\leq \sum_{s=1}^{n^*}\binom{n}{s}(1-p)^{T_{s,k}}\nonumber\\
&\leq \sum_{s=1}^{n^*}\brac{\frac{ne}{s}\cdot
 \exp\set{-p\b_k \frac{n^{k-1}}{(k-1)!} \ooi}}^s\label{smallS}\\
&\leq \sum_{s=1}^{n^*}n^{-(A+1)s}= O\bfrac{1}{n^{A+1}}\nonumber.
\end{align}
We now use \eqref{relate} to transfer this bound to $\bA_{n,m;k}$.

{\bf Proof of part (ii).}\\
 Let $m=n(\log n+c_n)/k$. Assume that $c_n\to c$.
We first observe that if $Z_s$ denotes the number of sets of $s=O(1)$ empty rows then
\mult{empty}{
\E(Z_s)=\binom{n}{s}\frac{\binom{\binom{n-s}{k}}{m}}{\binom{\binom{n}{k}}{m}}=\binom{n}{s} \prod_{i=0}^{m-1}\frac{\binom{n-s}{k}-i}{\binom{n}{k}-i}=\binom{n}{s} \bfrac{\binom{n-s}{k}}{\binom{n}{k}}^m \brac{1+O\bfrac{m^2}{n^k}}\\
\approx \frac{n^s}{s!}\cdot \prod_{i=0}^{k-1}\brac{1-\frac{s}{n-i}}^m=\frac{n^s}{s!}\cdot \prod_{i=0}^{k-1}\exp\set{-\frac{ms}{n}+O\bfrac{m}{n^2}}\approx \frac{n^s}{s!}e^{-skm/n} \approx \frac{e^{-cs}}{s!}.
}
The method of moments implies that $Z_1$ is asymptotically Poisson with mean $e^{-c}$ and so
\beq{good-zero}{
\Pr(Z_1=0) \approx e^{-e^{-c}}.
}
Going back to \eqref{smallS} with $p=(\log n+c_n)/\binom{n-1}{k-1}$ we see that we only need to consider $2\leq s\leq 4n^{1-\b_k}$. For these values of $s$, $T_{s,k}$ is bounded below by $s\binom{n-s }{k-1}\approx s\binom{n-1}{k-1}$. Thus we can bound the RHS of \eqref{smallS} from above by
$$
\sum_{s=1}^{4n^{1-\b_k}}\brac{\frac{3n}{s}\cdot\exp\set{-p\binom{n-1}{k-1}}}^s =\sum_{s=1}^{4n^{1-\b_k}}\bfrac{O(1)}{s}^s.
$$
Thus,
\beq{bigS}{
\Pr(\exists S,\log\log n\leq |S|\leq 4n^{1-\b_k}:\cB_S\text{ occurs})\leq
\sum_{s=\log\log n}^{n^{1-\b_k}}\bfrac{O(1)}{s}^s=o(1).
}
Finally we consider $2\leq s\leq L=\log\log n$. The final step is to prove (w.h.p) that when $p=(\log n+c_n)/\binom{n-1 }{k-1}$, $c_n \ra c$ constant the only obstruction to $\rank(\bA_{n,p;k})=n^*$ is the existence of empty rows ($Z_1>0$).

Given a set $S$, the number of choices of column that have an odd number of non-zero entries in $S$ (Type A columns) is given by $T_{s,k}$ above, and the number of choices of columns that have an even number of non-zero entries in $S$ (Type B columns) is
\[
R_{s,k}=\binom{s }{ 2}\binom{n-s }{ k-2}+ \cdots + \binom{s }{ k-1}(n-s).
\]
For $s \le L$, $R_{s,k} \le s^2 n^{k-2}$. The expected number $\mu_s$ of sets $S$ with no Type A columns and at least one Type B column is
$$\mu_s= \binom{n}{s}\brac{1-(1-p)^{ R_{s,k}}}(1-p)^{T_{s,k}}\le \frac{n^s}{s!}(pR_{s,k})\;e^{-ps\binom{n-1}{ k-1}\ooi }= O\bfrac{\log n}{n} e^{-cs}.$$
Thus, for constant $c$,
\beq{2n}{
\sum_{s=2}^{L}\mu_s=o(1).
}
Thus w.h.p. there is no set of $2\leq s\leq \log\log n$ rows where the dependency does not come from the rows all being  zero.
\proofend
\subsection{Proof of Theorem \ref{th2a}}
Because $c$ in \eqref{good-zero} is arbitrary and having a zero row is a monotone decreasing event, we can see that if $m_0=n(\log n-\log\log n)/k$ then $Z_1=Z_1(m_0)>0$ w.h.p. The reader can easily check that equations \eqref{bigS} and \eqref{2n} continue to hold. It follows that w.h.p. the rank of $\bM_{m_0}$ is $n^*-Z_1$. It then follows that $m_1=m^*$ if we never add a column that reduces the number of non-zero rows by more than one. Now \eqref{good-zero} implies that the expected number of zero rows in $\bM_{m_0}$ is $O(\log n)$ and so $Z_1\leq \log^2n$ w.h.p. So given this, the probability we add add a column that reduces the number of non-zero rows by more than one in the next $O(n\log n)$ column additions, is $O(n\log n \times ((\log^2n)/n)^2=o(1)$.
\section{Minimum Weight Basis}\label{mwb}
The first task here is to prove \eqref{Int1}. Let $B_{n,k}$ denote a minimum weight basis and let $W_{n,k}$ denote its weight. For a given a real number $X$ we can write
\[
X= \int_{p=0}^{X} dp = \int_{p=0}^{1} 1_{p \le X}dp.
\]

Thus
\begin{align}
W_{n,k}&=\sum_{\bc\in B_{n,k}}X_{\bc}\nonumber\\
&=\sum_{\bc\in B_{n,k}}\int_{p=0}^11_{p \le X_\bc}dp\label{qaz}\\
&=\int_{p=0}^1\sum_{\bc\in B_{n,k}}1_{p \le X_\bc}dp\nonumber\\
&=\int_{p=0}^1|\set{\bc\in B_{n,k}: p \le X_\bc }|dp\nonumber\\
&=\int_{p=0}^1(n^*-\rank(\bA_{p}))dp.\label{greedy}
\end{align}
Here $\bA_{p}$ is any matrix made up of those columns $\bc\in\Omega_{n,k}$ with $X_\bc\leq p$. And let $A_p$ denote the corresponding hypergraph.

{\bf Explanation for \eqref{greedy}:} Finding a minimum cost basis $B$ can be achieved via a {\em greedy algorithm}. We first order the columns of $\Omega_{n,k}$ as $\bc_1,\bc_2,\ldots,\bc_N, N=\binom{n}{k}$ in increasing order of weight $X_\bc$. Treating $B$ as a set of columns, we initialise $B=\emptyset$, and for $i=1,2,\ldots,N$ add $\bc_i$ to $B$ if it is linearly independent of the columns of $B$ selected so far. This means that for any $0\leq p\leq 1$, the number of columns in $B$ with $X_\bc > p$ must be equal to the co-rank of the set of columns selected before them i.e $B_p=\set{\bc\in B:X_\bc\leq p}$.
We claim that $B_p$ is a maximal linear independent subset of the columns of $\bA_{p}$. If it were not maximal, then another column of $\bA_{p}$ would have been added to $B_p$ by the greedy algorithm.

We obtain $\E W_{n,k}$ in \eqref{Int1} by taking the expectation of \eqref{greedy},  using Fubini's theorem to take the expectation inside the integral.

We first argue that
\beq{gt}{
\E(W_{n,k})=\Omega(n^{-(k-2)}).
}																
Let $\bc=(c_1,...,c_n)$, where $c_i\in\{0,1\}$ denotes the $i$-th coordinate of $\bc$. We can bound $W_{n,k}$ from below by $\sum_{i=1}^n\min\set{X_\bc: c_i=1}$. Let $N=\binom{n}{k}$. The number of ones in a fixed row of $\bA_{n,N;k}$ is $L=Nk/n$. The expected minimum of $L$ independent uniform $[0,1]$ random variables is $1/(L+1)$. Hence
$$\E(W_{n,k})\geq \frac{n^2}{k\binom{n}{k}+n}$$
and \eqref{gt} follows.

We next observe that for $c$ large we have
\beq{sizex}{
1-2ke^{-c}\leq x\leq 1.
}
Indeed, putting $x=1-y$ we have $(1-y)^{1/(k-1)}=1-e^{-c(1-y)}$. We see that if $f(y)=(1-y)^{1/(k-1)}-(1-e^{-c(1-y)})$ then $f(0)>0$ and $f(2ke^{-c})<0$ for large $c$.

Thus for $c$ large we have
\beq{clarge}{
\frac{c}{k}-\frac{cx^{k/(k-1)}}{k}+(1-e^{-cx}(1+cx))\geq 1-e^{-cx}(1+cx)\geq 1-e^{-99c/100}.
}
Fix some small $\e>0$ and let
\begin{equation}\label{Ceps}
c_\e=2\log 1/\e.
\end{equation}
It follows from Theorem \ref{th2}(i) with $A=k$, $p=km/(n {n-1 \choose k-1})$. and \eqref{Int1} that
\begin{align}
\E(W_{n,k})&\approx \int_{p=0}^{k!\g n^{1-k}\log n}(n^*-\E(\rank(\bA_{p})))dp\nonumber\\
&=\frac{(k-1)!}{n^{k-1}}\int_{c=0}^{k\g \log n} (n^*-\E(\rank(\bA_{c(k-1)!/n^{k-1}})))dc\nonumber\\
&=(I_1+I_2+I_3)\frac{(k-1)!}{n^{k-1}},\label{Eint}
\end{align}
where
$I_1=\int_{c=0}^{c_k^*}\cdots dc$ and $I_2=\int_{c_k^*}^{c_\e}\cdots dc$ and $I_3=\int_{c_\e}^{k\g \log n}\cdots dc$.

Since $H_{c/n^{k-1}}$ q.s. has $m \approx cn/k$ edges, it follows from Theorem \ref{th1} that
\beq{I1}{
I_1\approx \int_{c=0}^{c_k^*}\brac{n^*-\frac{cn}{k}}dc\approx c_k^*n\brac{1-\frac{c_k^*}{2k}}.
}
On the other hand, using the expression for rank from Corollary \ref{cor1}, with $x^{1/(k-1)}=(1-e^{-cx})$ substituted from \eqref{cstar}.
\begin{align}
I_2& \approx n\int_{c_k^*}^{c_\e} \brac{1-\brac{\frac{c}{k}-\frac{cx^{k/(k-1)}}{k}+(1-e^{-cx}(1+cx))}}dc \label{In2}\\
&=n\int_{c_k^*}^{\infty}\brac{ e^{-cx}(1+cx(k-1)/k)-\frac{c}{k}(1-x)}dc+A_\e,
\label{In2a}
\end{align}
where
\beq{Aeps}{
  |A_\e|=n\int_{c_\e}^{\infty}\brac{ e^{-cx}(1+cx(k-1)/k)-\frac{c}{k}(1-x)}dc\leq n\int_{c_\e}^{\infty} e^{-99c/100}dc\leq 2\e^2n.
  }
Theorem \ref{th1} as stated holds for $m=O(n)$, and thus cannot be used directly to estimate rank when $m/n \rai$. For $I_3$ we recall that $\bC_2=\bC_2(c)$ denotes the sub-matrix of $\bA_{c(k-1)!/n^{k-1}}$ induced by the edges of the 2-core. We then write
\beq{I3}{
I_3\leq \int_{c_\e}^{k\g \log n}(n^*-\E(\rank(\bC_2(c))))dc.
}
We first check the size of $|C_2|$ for $c=c_\e$. It follows from \eqref{sizex} that for $c$ large,
$$x^{1/(k-1)}-cx+cx^{k/(k-1)}=x^{1/(k-1)}-cxe^{-cx}\geq 1-e^{-99c/100}.$$
So, for $c=O(1)$ and large we have from \eqref{sizeC2} that w.h.p.
$$|C_2|\geq (1-o(1))n(1-e^{-99c/100}).$$
Let $m_\e=c_\e n/k$.  If we add an edge $e$ with one vertex not in $C_2$ and the remaining vertices in $C_2$ then the rank of $\bC_2$ goes up by one. Denote this event by $\cA_e$. Let $\bC^*=\bC^*(t)$  denote the following submatrix of $\bC_2$ at the time the number of columns is $m_\e+t$. We let $\bC^*(0)=\bC_2(c_\e)$ and we add the column corresponding to $e$ to $\bC^*$ only if $\cA_e$ occurs. Let $X_t$ denote the rank of $\bC^*(t)$, and let $Y_t=n^*-X_t$. Note that $X_t$ is equal to $\rank(\bC_2(c_\e))$ plus the number of columns in $\bC^*(t)$ that are not in $\bC_2(c_\e)$, and that $X_t\leq \rank(\bA_{m_\e+t})$. Note also that $|\rank(\bA_{m_\e+t})-\rank(\bA_{n,p_t,k})|\leq n^{2/3}$ where $p_t=(m_\e+t)/\binom{n}{k}$. Using \eqref{Ceps} we have that $Y_0\leq(1+o(1))ne^{-99c_\e/100}\leq 2\e^2n$. Now,
\beq{Be}{
\Pr(\cA_e)=\frac{Y_t\binom{n-Y_t}{k-1}}{\binom{n}{k}}\geq \frac{kY_t}{2n}
}
and so
\beq{Yt}{
\E(Y_{t+1}\mid Y_t)\leq Y_t-\frac{kY_t}{2n}.
}
Let $h=n^{1/2}$ and $u_r=Y_{rh}$. Assume that $n^{9/10}\leq Y_t\leq Y_0 $. It follows from \eqref{Be} and Hoeffding's Theorem \cite{Hoef} that q.s.
$$u_{r+1}\leq u_r-\frac{kh}{3n}u_r=\brac{1-\frac{kh}{3n}}u_r$$
and so q.s.
\beq{ur}{
u_r\leq \brac{1-\frac{kh}{3n}}^ru_0.
}
Going back to \eqref{I3} we can see that
\beq{I3a}{
I_3\leq O(n^{9/10})+\frac{hu_0}{n}\sum_{r=0}^\infty\brac{1-\frac{kh}{3n}}^r=
O(n^{9/10})+\frac{3u_0}{k}.
}
Here the final $O(n^{9/10})$ term accounts for only using \eqref{Yt} for $Y_t\geq n^{9/10}$ and for the errors of size $O(n^{2/3})$ introduced in the $m$ model versus the $p$ model of our matrix, see \eqref{relateH}, \eqref{relate}.

It follows from \eqref{I1}, \eqref{In2}, \eqref{Aeps} and \eqref{I3a} that $I_1+I_2+I_3$ are within $O(\e^2n)$ of what is claimed in the theorem. But $\e$ can be made arbitrarily small and the theorem follows.

\subsection{ Bounds for finite $k$}
We begin by estimating $c_k^*$. Let $x$ be as in \eqref{cstar}, then going back to the definition \eqref{defck}, we can determine the value of $c_k^*=c(x)$ from
\begin{equation}\label{eq0}
c \bfrac{k-1}{k} x^{\frac{k}{k-1}}-cx+x^{\frac{1}{k-1}}=0.
\end{equation}
Solve for $c$, and put $y=x^{1/(k-1)}$ to give
\begin{eqnarray}\label{cval}
c&=& \frac{1}{y^{k-2}-((k-1)/k)y^{k-1}}.
\end{eqnarray}
Substituting for $c$ via \eqref{cstar} gives
\begin{eqnarray}\label{yyy}
y=1-\exp \set{-\frac{ky}{k-(k-1)y}}.
\end{eqnarray}
If $x \in (0,1)$  then $y \in (0,1)$, and $y \ge x$. We look for solutions of the form
$y=1-z$. Making this substitution \eqref{yyy} becomes $z=q(z)$ where
\[
q(z)=\exp \set{ - \frac{k(1-z)}{1+(k-1)z}}.
\]
Let
\begin{equation}\label{zvalu}
z = z(\delta)=\frac{\delta}{k-(k-1)\delta},
\end{equation}
then (stretching notation somewhat) $q(\delta)=e^{-k(1-\delta)}$. Consider $f(\delta)=z(\delta)-q(\delta)$, then
\[
f(\delta) \ge \frac{\delta}{k}\brac{1+\frac{k-1}{k}\delta}- e^{-k}e^{k\delta}.
\]
Substitute $\delta=\th ke^{-k}$ to give
\[
f(\th) \ge e^{-k} \brac{ \th(1+\th(k-1)e^{-k})-e^{\th k^2e^{-k}}}.
\]
The function $k^2e^{-k}$ in the exponent of the last term is monotone deceasing for $k \ge 2$. Let $\th=3/2$, then for $k \ge 4$, it can be checked that $f(\th,k)>0$. Now $f(0)<0$ and so there is a solution  to $f(\delta)=0$ in the interval $(0,  \th ke^{-k})$.

Substitute $y=1-z$ into \eqref{cval} to obtain
\beq{cek}{
\frac ck=\frac{1}{(1-z)^{k-2}(1+(k-1)z)}
}
\begin{lemma}\label{cvalu}
\begin{enumerate}[(i)]
\item
Let   $\th=3/2$, then  for $k \ge 4$,
\begin{equation}\label{cvals}
k(1-\th e^{-k}) \le c_k^* \le k.
\end{equation}

\item For $k=3$, $c^*_3=2.753699...$.
\item If $k\geq 4$ and $c\geq c_k^*$ then the solution $x$ to \eqref{cstar} satisfies $x\geq 1-3ke^{-c}/2$.
\end{enumerate}

\end{lemma}
\begin{proof}
(i) For the upper bound we note that for $k \ge 3$ the denominator of $c$ in \eqref{cek} is
monotone increasing for $z \le 1/(k-1)^2$ from a value of one when $z=0$. For the lower bound, as
$1/(1-z)^{k-2} > 1+(k-2)z$, it follows from \eqref{cek}, the definition of $z$ in \eqref{zvalu}, and $\delta < \th ke^{-k}$ that
\[
\frac{c}{k} > \frac{1+(k-2)z}{1+(k-1)z}  = 1-\frac{\delta}{k} > 1-\th e^{-k}.
\]

(ii) Set $y=\sqrt{x}$ and invert \eqref{cstar} to obtain
\[
c=\frac{1}{y^2} \log \frac{1}{1-y}.
\]
Inserting this into  \eqref{cval} gives
\[
y+\brac{\frac23 y-1}\log\frac{1}{1-y}=0.
\]
This was solved numerically to give the following results for $y,x,c^*_3 $
\begin{equation}\label{values}
y=0.8833916,\quad x=0.9398891,\quad c^*_3=2.753699.
\end{equation}

(iii)

Let $x=1-\e$. We first verify that $\e\leq 1/c$. Putting $f(\e)=1-\e-(1-e^{-c+c\e})^{k-1}$ we see that $f(0)>0$ and $f(1/c)<0$ for $c\geq c_k^*$ as given in (i).   If $ay<1$, then $1-(1-y)^a < ay$.
As $(k-1)e^{-c+c\e}<1$ for any $\e<1-(\log(k-1))/c$,
$$f(c^{-1})=1-c^{-1}-(1-e^{-c+1})^{k-1}\leq 1-c^{-1}-1+(k-1)e^{-c+1}.$$
Now $c(k-1)e^{-c+1}$ is decreasing as a function of $c$. And for $k\geq 4$, $k(k-1)e^{-c+1}$ and $e^{(3k/2)e^{-k}}$ are decreasing as functions in of $k$. Therefore, for $c$ satisfying \eqref{cvals},
\[
c(k-1)e^{-c+1}< k(k-1)e^{-(k-1)}e^{(3k/2)e^{-k}} <1.
\]

 Let $x=1-\e$, and $\delta=e^{-c+c\e}$. Rewrite \eqref{cstar} as
\beq{ceps}{
-\log(1-\e)=\e+\frac{\e^2}2+\cdots=(k-1)\brac{\delta+\frac{\delta^2}2+\cdots}.
}
It must hold that $\e \le (k-1)\delta$ otherwise the left hand side is greater than the right hand side. Thus, as $\e <1/c$,
\[
\e\leq (k-1)e^{-c+c\e}\leq (k-1)e^{-c+1}.
\]
A repeated application of this bound, \eqref{cvals}  and direct calculation gives
$$
\e\leq (k-1)\exp\set{-c+(k-1)ce^{-c+1}}\leq (k-1)\exp\set{-c+(k-1)ke^{1-(1-\th e^{-k})k}}\leq  3ke^{-c}/2.$$
\end{proof}

Going back to \eqref{I1} and using Lemma \ref{cvalu}(i), we see that for $k\geq 4$,
\beq{I1a}{
\frac{kn}{2}\brac{1-\frac{9}{4}e^{-2k}} \leq I_1 \leq \frac{kn}{2}.
}

We evaluate $I_2$ from \eqref{In2}--\eqref{In2a} in two parts. Firstly, using Lemma \ref{cvalu}(iii) for $c\geq c_k^*$,
\beq{missed}{
-\frac32 ce^{-c} \le -\frac{c}{k}(1-x) \le 0.
}
Note also that $1-3ke^{-c}/2 \ge 1-1/2k$ for $k \ge 4$ and $c\ge c^*_k$. Thus
\[
e^{-c}\brac{1+c \frac{(k-1)(2k-1)}{2k^2}} \le e^{-cx}\brac{1+cx\frac{k-1}{k}} \le
e^{1/2}e^{-c}\brac{1+\frac{c(k-1)}{k}}.
\]
For the LHS we replace $e^{-cx}$ by $e^{-c}$ (since $x\leq 1$) and $x$ by $1-1/2k$. For the RHS we replace $cx(k-1)$ by $c(k-1)$, and $e^{-cx}=e^{-c+c\e}$. Using Lemma \ref{cvalu}(i) and (iii), as $c^*>1$, it follows that
\beq{included}{
e^{c\e} \le e^{ (3k/2)ce^{-c}}  \le e^{ (3k/2)c^*e^{-c^*}} \le e^{1/2}.
}
Adding the contributions from \eqref{missed} and \eqref{included} we find that
\[
n\int_{c^*}^\infty e^{-c}\brac{1-c \frac{k^2+3k-1}{2k^2}} \;dc \le \; I_2 \; \le n e^{1/2}\int_{c^*}^\infty e^{-c}\brac{1+\frac{c(k-1)}{k}} \; dc.
\]
Thus, with the {\em indefinite integral} $\int e^{-c}(1+Ac)=-e^{-c}(1+A+Ac)$, we get
\[
ne^{-c^*_k} \brac{\frac{k^2-3k+1}{2k^2}-c^*_k \frac{k^2+3k-1}{2k^2}} \le
\;I_2\; \le n e^{1/2}e^{-c_k^*}\brac{\frac{2k-1}{k}+c^*_k \frac{k-1}{k}},
\]
or more simply
\[
-n \frac{k}{2}e^{-c^*_k}\brac{1+\frac{3}{k}}\le\; I_2 \;\le n\frac{k}{2}e^{-c^*_k}\;2e^{1/2}\brac{1+\frac{1}{k}}.
\]
Noting that $e^{-c^*_k}\le 6e^{-k}/5$ for $k \ge 4$, we have
\[
n \frac{k}{2}\brac{1-\frac94 e^{-2k}-\frac{21}{10} e^{-k}} \le I_1+I_2 \le
n \frac{k}{2}\brac{1+3e^{1/2}e^{-k}}.
\]
Thus, for some $\e_k$, $|\e_k| \le 5$,
\[
I_1+I_2= n \frac{k}{2}\brac{1+\e_ke^{-k}}.
\]

\section{Open questions}
\begin{enumerate}[{\bf Q1}]
\item The formula for the cost of a minimum weight basis when $k\geq 3$ given by Theorem \ref{WTF} is asymptotically accurate, but lacks the elegance of the case where $k=2$. Can the expression be simplified for say, $k=3$?
\item The $\z(3)$ result of \cite{F1} was generalised quite substantially to consider minimum weight spanning trees of $d$-regular graphs, when $d$ is large, see \cite{BFM}. In the context of $\bA_{n,m;k}$, this suggests that we consider the case where each row has exactly $d$ ones. Here we can study the rank as well as $W_{n,k}$.
\end{enumerate}


\end{document}